\documentclass{amsart}
\usepackage{amssymb}
\usepackage{amsmath}
\usepackage{amsfonts}
\usepackage{amssymb}
\usepackage{graphicx}
\usepackage{color}
\newtheorem{mainthm}{Theorem}[]
\newtheorem{thm}{Theorem}[section]

\newtheorem*{thm*}{Theorem} 
\newtheorem{cor}[thm]{Corollary}
\newtheorem{lem}[thm]{Lemma}

\newtheorem{prop}[thm]{Proposition}

\theoremstyle{definition}

\newtheorem{defin}[thm]{Definition}\newtheorem{mainrem}[mainthm]{Remark}
\newtheorem*{rems*}{Remarks}
\theoremstyle{remark}
\newtheorem{rem}[thm]{Remark}

\newcommand{\CP}{\mathbb{C\mkern1mu P}}               
\newcommand{\C}{\mathbb{C}}

\newcommand{\R}{\mathbb{R}}

\newcommand{\Hess}{\mathrm{Hess}}

\newcommand{\Sph}{\mathbb{S}}
\newcommand{\Par}{\mathrm{Par}}

\DeclareMathOperator{\GL}{GL}
\DeclareMathOperator{\Ric}{Ric}\DeclareMathOperator{\ad}{ad}

\DeclareMathOperator{\scal}{scal}

\DeclareMathOperator{\Or}{O}\DeclareMathOperator{\Iso}{Iso}
\DeclareMathOperator{\SO}{SO}

\newcommand{\gl}{\mathfrak{gl}}\newcommand{\su}{\mathfrak{su}}

\newcommand{\so}{\mathfrak{so}}
\newcommand{\Lg}{\mathfrak{g}}
\newcommand{\Lu}{\mathfrak{u}}

\DeclareMathOperator{\Fix}{Fix}
\DeclareMathOperator{\trace}{tr}
\DeclareMathOperator{\Ad}{Ad}
\DeclareMathOperator{\Har}{Hk}\DeclareMathOperator{\Ham}{Hm}
\DeclareMathOperator{\Rm}{Rm}

\DeclareMathOperator{\pr}{pr}

\DeclareMathOperator{\id}{{id}}

\DeclareMathOperator{\rank}{rank}

\newcommand{\bml}{\bigl\langle} 
\newcommand{\bmr}{\bigr\rangle}     
\newcommand{\ml}{\langle}                     
\newcommand{\mr}{\rangle}                     

\newcommand{\bx}{\bar{x}}

\newcommand{\grad}{\mathrm{grad}}
\newcommand{\eps}{\varepsilon}

\newcommand{\bv}{\bar{v}}

\newcommand{\iso}{\mathfrak{iso}}

\newcommand{\gU}{\mathsf{U}}
\newcommand{\G}{\mathsf{G}}
\begin{document}

\title[Invariant nonnegativity conditions]{ A Lie algebraic approach to 
Ricci flow invariant curvature conditions and Harnack inequalities}

\author{Burkhard Wilking}

\maketitle

\begin{abstract} We consider a subset $S$ of the complex Lie algebra $\so(n,\C)$
and the cone $C(S)$ of curvature operators which are nonnegative on $S$. We show that 
$C(S)$ defines a Ricci flow invariant curvature condition if $S$ is invariant under 
$\Ad_{\SO(n,\C)}$. The analogue for K\"ahler curvature operators holds as well. 
Although the proof is very simple and short it recovers all previously known 
invariant nonnegativity conditions. 
As an application we reprove 
 that a compact K\"ahler manifold with positive orthogonal bisectional curvature 
evolves to a manifold with positive bisectional curvature and is thus biholomorphic to $\CP^n$. 
Moreover, the methods can also be applied to prove Harnack inequalities.
\end{abstract}

We consider a Lie algebra $\Lg$ endowed with a scalar 
product $\ml \cdot,\cdot \mr$ which is invariant under the adjoint 
representation of the Lie algebra. 
The reader should think of $\Lg$ either as the space $\so(n)$ 
of skew adjoint endomorphism of $\R^n$ with 
 the scalar product $\ml A,B\mr =-\tfrac{1}{2}\trace(AB)$ 
or of the Lie subalgebra $\Lu(n)\subset \so(2n)$ corresponding to the unitary group $\gU(n)\subset \SO(2n)$
endowed with the induced scalar product.

We consider the space of selfadjoint endomorphisms 
of $S^2(\Lg)$.
Every selfadjoint endomorphism 
$R\in S^2(\Lg)$  is determined by 
the corresponding bilinear form
$(x,y)\mapsto \ml Rx,y\mr$. 
The extension of this form 
to a complex bilinear form 
\[
R\colon \Lg\otimes_\R\C \times \Lg\otimes_\R\C\rightarrow \C
\]
will be denoted with the same letter $R$. 
Notice that for any $x\in \Lg\otimes_{\R}\C$ the number $R(x,\bx)$ is real, 
where $x\mapsto \bx$ is complex conjugation.

Also recall that the space of algebraic curvature operators $S^2_B(\so(n))$ 
is a linear subspace of $S^2(\so(n))$. Similarly the space of algebraic 
K\"ahler curvature operators $S^2_K(\Lu(n))$ is a linear subspace of $S^2(\Lu(n))$.
The subspaces are also invariant  under the Ricci flow ODE on $S^2(\Lg)$
\[
 R'=R^2+R^\#
\]
where $\ml R^\#x,y\mr=-\tfrac{1}{2}\trace(\ad_xR\ad_yR)$ for $x,y\in \Lg$.
We have the following basic result
\begin{mainthm}\label{mainthm: ricci} Let $S$  be a subset of the complex Lie algebra $\Lg\otimes_\R\C$ and let $\G_\C$ denote 
a  Lie group with Lie algebra $\Lg\otimes_\R\C$ . If $S$ is invariant under the adjoint representation 
of $\G_{\C}$, then for $h\in\R$  the set 
\[
C(S,h):= \{ R\in S^2(\Lg)\mid R(v,\bar v)\ge h\mbox{ for all $v\in S$} \} 
\]
is invariant under the ODE $R'=R^2+R^\#$.
\end{mainthm}

In many cases $S$ is scaling invariant and then $h=0$ is the only meaningful 
choice. For $h=0$ the set  $C(S):=C(S,0)$ is a cone and the curvature condition $C(S)$ 
can be thought of as a nonnegativity condition. 
We recall that for a $O(n)$-invariant subset $C\subset S^2(\so(n))$ we say that a manifold satisfies $C$ 
if the curvature operator at each point is in $C\cap S^2_B(\so(n))$. 
Although the proof of the theorem is just a few lines long its statement recovers
via Hamilton's maximum principle [1986] the  invariance of all previously known 
invariant nonnegativity  conditions: \\
\begin{mainrem}
\label{mainrem} In all of the following examples we assume $h=0$. 
\begin{enumerate}
\item[a)] In the case of  $S=\so(n,\C)$  the theorem 
recovers the invariance of the cone $C(S)$ of nonnegative operators -- a result due to Hamilton.
\item[b)] In the case of  $\Lg=\so(n,\R)$ and $S=\{X\in\so(n,\C)\mid \trace(X^2)=0\}$  the theorem 
recovers the invariance of the cone $C(S)$ of $2$ nonnegative operators. 
 This result is also due to Hamilton.
\item[c)]  The invariance of nonnegative isotropic curvature, which was shown independently by Nguyen [2007,2010] and
 Brendle and Schoen [2009],
 can be seen
by setting $\Lg=\so(n)$ and
\[
 S:=\{X\in \so(n,\C)\mid \rank(X)=2, X^2=0 \}.
\]
The equation $X^2=0$ is equivalent to saying that each vector $v$ 
in the image of $X$ is isotropic, i.e., the imaginary part and the real part are perpendicular and have the same norm.
It is easy to see that $C(S)\cap S^2_B(\so(n))$ is indeed the space of curvature operators 
with nonnegative isotropic curvature.

\item[d)] The invariance of the condition that the manifold crossed with $\R$ has nonnegative isotropic curvature 
due to Brendle and Schoen corresponds to
\[
 S:=\{X\in \so(n,\C)\mid \rank(X)=2, X^3=0 \}.
\]
 
\item[e)] The invariance of the condition that the manifold crossed with $\R^2$ has nonnegative isotropic curvature
due to Brendle and Schoen
corresponds to
\[
S:=\{X\in \so(n,\C)\mid \rank(X)=2\}.
\]
It was then observed by Ni and Wolfson [2008] that $M$ satisfies $C(S)$ if and only if  
$M$ has nonnegative complex curvature. Ni and Wolfson also 
gave a simpler proof that positive complex curvature is invariant under the Ricci flow. 
For the author this simplification was one indication that 
proofs should be simpler in the complex setting.

This invariance was the key new result in the proof of the differentiable quarter pinched sphere theorem 
of Brendle and Schoen [2009].
The convergence of the metric under Ricci flow toward constant curvature then followed from [B\"ohm and Wilking, 2008],
see also subsection~\ref{subsec: pic1}.

\item[f)] The invariance of nonnegative bisectional curvature due to Mok [1988]
can be recovered from the theorem as well.
The Lie algebra $\Lu(n)\otimes_{\R}\C$ 
can be naturally identified with the algebra of complex $n\times n$ matrices $\gl(n,\C)$.

If we put 
\[
 S=\{X\in \gl(n,\C) \mid \rank(X)=1\},
\]
one can check by straightforward computation
that $C(S)\cap S^2_K(\Lu(n))$ is given by the cone of K\"ahler curvature operators 
with nonnegative bisectional curvature.
We would like to emphasize that Mok's proof of the invariance used a second variation argument 
for the first time in this context. The proof of the invariance of nonnegative isotropic curvature 
by Nguyen [2007,2010] and  Brendle and Schoen [2009] also relied on second variation. The same is true for the proof 
of Theorem~\ref{mainthm: ricci}. 
\item[g)] The theorem also shows the invariance of orthogonal bisectional curvature, if we put 
\[
S:=\{X\in \gl(n,\C) \mid \rank(X)=1, X^2=0\}.
\]
The invariance was announced by Hamilton and H.D.~Cao in the early 90s and a proof was given by Gu and Zhang [2010].
\end{enumerate}

\end{mainrem}

The theorem can also be generalized to obtain Harnack inequalities:
Let $(M,g(t))$ be a solution to the Ricci flow. 
We endow the Lie algebra $\Lg(p,t)$ of the isometry group of 
$\Iso(T_pM,g(t))$ with a scalar product $((A,v),(B,w))=-\tfrac{1}{2}\trace(AB)+ g(t)(v,w)$
for skew adjoint endomorphisms $A,B$ of $(T_pM,g(t))$ and $v,w\in T_pM$. 
The Harnack operator $\Ham$ can be viewed as a self adjoint endomorphism of $\Lg(p,t)$. 
As a consequence of Hamilton's work [1993] the Harnack operator satisfies (cf. section~\ref{sec: har}) with respect to moving frames an evolution equation of the form
\[
 \Ham'=\Delta \Ham+ 2(\Ham \pr \Ham+ \Ham^\#)+\tfrac{2}{t}\Ham
\]
Here $\pr\colon \Lg(p,t)\rightarrow \so(T_pM)$ denotes the orthogonal 
projection and 
\[\ml \Ham^\#(x),y\mr =-\tfrac{1}{2}\trace(\ad_{\cdot}^{tr}x\Ham \ad_{\cdot}^{tr}y\Ham),\] 
where $\ml \ad_z^{tr}x,w\mr =\ml x,[z,w]\mr$ and $\ad_\cdot^{tr}x$ is the map 
$z\mapsto \ad_z^{tr}x$, which is easily seen to be skew adjoint with respect to $\ml\cdot,\cdot \mr$.

\begin{mainthm}\label{mainthm: har} Let $\Lg$ be the Lie algebra of $Iso(\R^n)$ endowed with the scalar product from above. 
Let $S$ be a subset of $\Lg_{\C}=\Lg\otimes_{\R}\C$. We consider $\Lg$ endowed 
with coadjoint representation $g\mapsto \Ad_{g^{-1}}^{tr}$, $\ml \Ad_a^{tr}v,w\mr= \ml v,  \Ad_aw\mr$. 
We suppose that $S$ is invariant under the natural extension of the coadjoint representation 
to a represention of $\SO(n,\C)\rtimes \C^n$. 
Then the cone 
\[
C(S)=\bigl\{\Ham\in S^2(\Lg)\mid \Ham(x,\bar x)\ge 0 \mbox{ for all $x\in S$}\bigr\}
\]
defines a Ricci flow invariant condition.
\end{mainthm}

It is not hard to see that the ODE $\Ham'=\Ham \pr \Ham+ \Ham^\#$ is equivariant 
with respect to the action of $\Iso(\R^n)$ on $S^2(\Lg)$ given by
$g\star \Ham:=\Ad_g\Ham \Ad_{g}^{tr}$.

The theorem recovers Brendle's  recent generalization of Hamilton's Harnack inequality 
by putting
\[
 S=\{(A,v)\mid A\in \so(n,\C),\rank(A)=2, v\in A(\C^n)\}
\]
As is shown in [Brendle, 2009] this still implies the usual trace Harnack inequality. \\[1ex]

A K\"ahler manifold $M$ is said to have positive orthogonal bisectional curvature if $K(v,w)+K(v,iw)> 0$
holds for all unit vectors $v,w\in T_pM$ with $\C\cdot v\perp \C\cdot w$, where $ K(v,w)$ denotes the
 sectional curvature of the plane spanned by $v$ and $w$. A K\"ahler surface has nonnegative orthogonal
 bisectional curvature if and only if 
it has nonnegative isotropic curvature. 
Thus orthogonal bisectional curvature is independent of the traceless 
Ricci part if $n=2$.
Furthermore, $M$ has nonnegative bisectional curvature if and only if 
$M\times \C$ has nonnegative orthogonal bisectional curvature.
We will give a somewhat simpler proof of the following theorem.

\begin{mainthm}\label{mainthm: bisec} 
A compact K\"ahler manifold of complex dimension $n>1$ with positive orthogonal bisectional curvature evolves under the Ricci flow
 to a manifold with positive bisectional curvature.
\end{mainthm}

The theorem is not new.
Chen [2007] shows that for a compact solution to the 
K\"ahler Ricci flow which has positive first Chern class and positive orthogonal 
bisectional curvature throughout space time, the 
bisectional curvature becomes positive.
Then Gu and Zhang [2010] show that indeed the first Chern class is positive and they 
also give a proof of the invariance of positive orthogonal bisectional curvature.

We decided to give a proof which is independent of [Chen, 2007] and Gu and Zhang [2010]. 
However, in our proof as well as in [Chen, 2007] 
  a key ingredient is a result of 
Perelman, written up by Sesum and Tian [2006],
 ensuring that for a compact K\"ahler manifold with positive first Chern class all non flat blow up limits  
are compact. 

Although we do not need it we should mention that
Chen, Sun and Tian [2009] gave a new  proof of the statement 
that a K\"ahler manifold with positive bisectional curvature 
evolves under the normalized K\"ahler Ricci flow to the Fubini study metric on $\CP^n$. 
The new proof does not need directly the solution
of the Frankel conjecture due to Mori [1979] and Siu and Yau [1980].

We will explain in an appendix why Brendle and Schoen's strong maximum principle [2008] 
carries over to our more general setting.  Therefore a K\"ahler metric on a compact 
manifold with nonnegative orthogonal bisectional curvature evolves under the Ricci flow to one  
with positive  orthogonal bisectional curvature unless the holonomy group 
is not equal to $\gU(T_pM)$.  Combining with Berger's classification of holonomy groups [1955]
and the solution of the Frankel conjecture, one can show that a locally irreducible compact K\"ahler manifold of dimension $n>1$ with nonnegative orthogonal bisectional curvature is either 
biholomorphic to $\CP^n$ or locally isometric to a hermitian symmetric space. 
This recovers a rigidity theorem of Gu and Zhang [2010] which in turn 
generalized a result of Mok [1988].

The paper is organized as follows. 
Section~\ref{sec: ricci} contains the proof of Theorem~\ref{mainthm: ricci}. 
Although the essential part of the argument in the proof of 
Theorem~\ref{mainthm: har} is completely analogous, we do need a little extra preparation, which is done in
 section~\ref{sec: har}. Here we provide a maximum principle for Harnack operators. 
One of the main points is to explain that the invariance group 
needed for the maximum principle is naturally isomorphic to $\Iso(T_pM)$.
This in fact is a simple consequence of the work of Chow and Chu [1995]
as well as Brendle [2009].
Although we do not need it we have 
added a subsection showing that in quite 
a few cases the maximum principle 
for Harnack operators can not possibly yield any meaningful outcome.
Using section~\ref{sec: har} the proof of Theorem~\ref{mainthm: har} is reduced to an ODE
problem, which is solved, completely analogously to section~\ref{sec: ricci} in section~\ref{sec: mainthm har}.
In section~\ref{sec: pinching} we show that often Theorem~\ref{mainthm: ricci} can be used 
to show that a nonnegativity condition pinches towards a stronger nonnegativity condition. 
We show for example that nonnegative orthogonal bisectional curvature pinches toward nonnegative 
bisectional curvature.
Section~\ref{sec: bisec} is devoted to the proof of Theorem~\ref{mainthm: bisec}.
An appendix is devoted to  the strong maximum principle.

This paper was written up while the author was visiting the University of California at Berkeley 
as a Visiting Miller Professor. I am grateful to the Miller institute for support and hospitality.
I would also like to thank Esther Cabezas-Rivas and Sebastian Hoelzel 
for pointing out several typos and inconsistencies in an earlier version of this
paper.

\section{Proof of Theorem~\ref{mainthm: ricci}.}\label{sec: ricci}
Let $S\subset \Lg\otimes_\R \C$ be invariant under $\Ad_{\G_\C}$ and put as in the theorem
\[
C(S,h):= \{ R\in S^2(\Lg)\mid R(v,\bar v)\ge h\mbox{ for all $v\in S$} \}. 
\]
Since $C(S,h)$ does not change if we replace $S$ by its closure we may assume that $S$ is closed.
As we will see below it suffices to show\\[1ex]
{\bf Claim.} If $R\in C(S,h)$ and   $v\in S$ 
with $R(v,\bar v)=h$, then $R^2(v,\bar v)+R^\#(v,\bar v)\ge 0$.\\[1ex] 
Clearly, $R^2(v,\bar v)\ge 0$. 
We plan to establish the inequality by showing that the second summand is nonnegative as well: 
\[
2R^\#(v,\bar v)=-\trace (\ad_v R\ad_{\bar v} R)\ge 0.
\]
Using that $S$ is invariant under $\Ad_{\G_{\C}}$ we deduce that
\[h\le R(\Ad_{\exp(t x)}v ,\Ad_{\exp(t \bar x)}\bar v)\]
 for all $x\in \Lg\otimes_{\R}\C$ and for all $t$ and with equality at $t=0$. 
Recall that $\Ad_{\exp(t x)}=\exp(t\ad_x)$. Thus differentiating twice 
with respect to  $t$ and evaluating at $0$ gives
\[
 0\le 2 R(\ad_x v ,\ad_{\bar x}\bar v) +R(\ad_x\ad_x v , \bar v)+R(v ,\ad_{\bar x}\ad_{\bar x}\bar v )
\]
If we replace $x$ by $i x$, it is easy to see that the first summand in the above inequality remains unchanged 
while the other two summands change their sign. 
Therefore 
\begin{eqnarray}\label{eq: nonneg}
 0&\le& R(\ad_x v ,\ad_{\bar x}\bar v)= R(\ad_v x ,\ad_{\bar v}\bar x) \mbox{ for all $x\in \Lg\otimes_{\R}\C$. }
\end{eqnarray}
In other words the hermitian operator 
$ -\ad_{\bar v} R\ad_v$ and its conjugate $-\ad_{v} R\ad_{\bar v}$ on the unitary vectorspace $\Lg\otimes_{\R}\C$ are nonnegative.
Recall that we plan to show 
$\trace (-\ad_v R\ad_{\bar v} R)\ge 0$. It is an elementray well known lemma  that 
the scalar product of two nonnegative hermitian matrices is nonnegative. 
By a slight extension of this lemma it suffices to show that the 
operator $R$ induces a nonnegative sesquilinear form on the image of the first operator
$ -\ad_v R\ad_{\bar v}$. 
Clearly the image  is contained in the image of 
$\ad_v$ and by \eqref{eq: nonneg} $R$ is indeed nonnegative on it which completes the proof of the claim.\\[2ex]
If $h=0$, then we may assume that 
$S$ is scaling invariant and the invariance of $C(S)$ follows immediately from the claim.
In general we have to be a bit more cautious since we do not know that 
the infimum of $\{R(\bar v, v)\mid v\in S\}$ is attained.

In order to see that the above claim is sufficient in the general case
we consider a solution $R(t)$ to the ODE $R'=X(R)=R^2+R^\#+\eps I$ for some $\eps>0$.
We plan to show that if $R(0)\in C(S,h)$ then $R(t)\in C(S,h-\eps t)$ for $t>0$. 
By taking the limit $\eps\to 0$ we get the desired result. 

Suppose, on the contrary, that $R(t_i)\notin C(S,h-\eps t_i)$ for some positive $t_i\to 0$.
Thus there are $v_i\in S$ with $R(t_i)(v_i,\bar v_i)<h-\eps t_i$. 
If $v_i$ stays bounded we can assume that $v_i\to v\in S$  
with $R(0)(v,\bar v)=h$. From the above claim $R'(0)(v,\bar v)\ge 0$.
Thus there is a neighborhood $U$ of $v$ and $\delta >0$ with $R'(t)(u,\bar u)\ge -\eps/2$ 
for $u\in U$ and all $t\in [0,\delta]$. Clearly this gives a contradiction. 

Thus we may assume $\|v_i\|\to \infty$.
 After passing to a subsequence, 
$\tfrac{v_i}{\|v_i\|}\to w$ with
$R(w,\bar w)\le 0$ and  
\[
w\in \partial_\infty S:=\{ Y\in \Lg_{\C}\mid\mbox{ there exists $\lambda_i\in \R$ and $v_i\in S$ with $\lambda_i\to 0$ and 
$\lambda_iv_i \to Y$}\}
\]
We call $\partial_\infty S$ the boundary of $S$ at infinity.
Clearly $\partial_\infty S$  is scaling invariant and invariant under $\Ad_{\G_{\C}}$ 
using $R\in C(S,h)$ it is elementary to check 
that $R\in C(\partial_\infty S,0)$, cf. Lemma~\ref{lem: pinching} below. 
In particular $R(w,\bar w)=0$ and from the above claim $(R^2+R^\#)(w,\bar w)\ge 0$. 
Therefore $R'(0)(w,\bar w)\ge \eps$. 
This in turn shows that there is a neighborhood $U$ of $w$ and $\delta>0$ such that 
$R'(t)(u,\bar u)>\eps/2$ for all $u\in U$ and $t\in [0,\delta]$. 
For large $i$ we have $v_i= \|v_i\| u_i$ for some $u_i\in U$ and therefore 
$R(t_i)(v_i,\bar v_i)\ge R(0)(v_i,\bar v_i)\ge h$ for all large $i$ -- a contradiction.

\begin{rem} In the case of $\Lg=\Lu(n)$, $\Lg_{\C}=\gl(n,\C)$ one can generalize the theorem slightly. 
For $h_1,h_2\in\R$ the set \[
C(S,h_1,h_2)=\{ R\in S^2(\Lu(n))\mid R(v,\bar v)+ h_2 \trace(v)\trace(\bv)\ge h_1, \mbox{ for all $v\in S$}\}\]
is invariant under the ODE as well, provided that $S$ is $\Ad_{\GL(n,\C)}$-invariant.
\end{rem}

\section{Maximum principle for Harnack operators.}\label{sec: har}

In this section we 
establish a maximum principle for Harnack operators 
which only needs the invariance under a group action of $\Iso(\R^n)$.
This is in fact a simple consequence of the work of Chow and Chu [1995],  see also 
[Chow and Knopf, 2002].

Let $(M,g(t))$ be a solution 
to the Ricci flow $t\in (0,T)$.
We consider $N=M\times [0,T)$.
We define a Riemannian metric $\ml \cdot,\cdot \mr$
on $N$ by 
\[
 \ml v,w\mr=g(t)(v,w),\,\,\, \bigl \ml v,\tfrac{\partial}{\partial t} \bigr\mr=0,\,\,\,\bigl \ml \tfrac{\partial}{\partial t},\tfrac{\partial}{\partial t} \bigr\mr=1
\]

We identify the Lie algebra $\iso(T_pM)$ of the isometry group of $T_pM$ 
with $ \Lambda^2 T_pM\oplus \{ \tfrac{\partial}{\partial t}\wedge v\mid v\in T_pM\}$
and define the Harnack operator $\Har$ as an selfadjoint endomorphism
of $\iso(T_pM)$ by 
\begin{eqnarray*}
\ml \Har(X\wedge Y), W\wedge Z\mr &=&R_{g(t)}(X\wedge Y,W\wedge Z)\\
\ml \Har(X\wedge Y), \tfrac{\partial}{\partial t}\wedge Z\mr &=&t(\nabla^{g(t)}_X\Ric)(Y,Z)-t(\nabla^{g(t)}_Y\Ric)(X,Z)\\
\ml \Har( \tfrac{\partial}{\partial t}\wedge X), \tfrac{\partial}{\partial t}\wedge Y\mr &=&
t^2(\Delta \Ric)_{g(t)}(X,Y)-
\tfrac{t^2}{2}\Hess_{g(t)}(\scal)(X,Y)\\
&&
+2t^2\sum_i\Ric_{g(t)}(e_i,e_i)\Rm^{g(t)}(e_i\wedge X, e_i \wedge Y)\\&&-t^2\Ric_{g(t)}(\Ric^{g(t)}X,Y)+\tfrac{t}{2}\Ric(X,Y)
\end{eqnarray*}
where $\Ric^{g(t)}$ resp. $\Ric_{g(t)}$ is the Ricci tensor of $(M,g(t))$ viewed as $(1,1)$ resp. $(2,0)$ tensor, 
 $e_i$ is an orthonormal basis of eigenvectors of $\Ric^{g(t)}$
and where  $\Hess_{g(t)}(\scal)$ is the Hessian of the scalar curvature of $(M,g(t))$.
By putting 
\[
 \Ham(X\wedge Y+ \tfrac{\partial}{\partial t}\wedge Z,X\wedge Y+ \tfrac{\partial}{\partial t}\wedge Z)=
\Har(X\wedge Y+ \tfrac{1}{t}\tfrac{\partial}{\partial t}\wedge Z,X\wedge Y+ \tfrac{1}{t}\tfrac{\partial}{\partial t}\wedge Z)
\]
we get back to the usual definition of the Harnack operator. 

 Let $\Lg$ be the Lie algebra of $\Iso(\R^n)$ 
endowed with the natural scalar product from the introduction.  
Consider on $\Lg$ the coadjoint representation 
\[
 \Iso(M,g)\rightarrow \GL(\Lg),\,\, g\mapsto \Ad_{g^{-1}}^{{tr}}
\]
Let $S^2(\Lg)$ denote the vectorspace of selfadjoint endomorphisms 
of $\Lg$ endowed the  representation of $ \Iso(\R^n)$ given by 
$g\star R= \Ad_gR\Ad_{g}^{tr}$ for $R\in S^2(\Lg)$ and $g\in \Iso(\R^n)$. 
Although it is not important for us, we mention that by Brendle [2009],  
the Harnack operator is always contained in a linear subspace $S^2_B(\Lg)$ 
of operators satisfying the first 
Bianchi identity. 

Recall that a family of sets $C(t)\subset V$ ($t\in (a,b)$) in a vectorspace $V$  is called 
invariant under a ODE $v' =f(v)$ if for any solution $v(t)$ ($t\in [t_0,s]$) 
with $v(t_0)\in C(t_0)$ we have $v(t)\in C(t)$ for $t\ge t_0$.
In this section we want to prove
\begin{thm}\label{thm: max} Suppose $C(t)\subset S^2_B(\Lg)$ is a family of  closed convex sets
 which is invariant 
under the above representation of $\Iso(\R^n)$. 
 We assume that $C$ is invariant under the ODE
\[
 \Har'=2(\Har \pr \Har+\Har^\#)
\]
where $\pr\colon \iso(\R^n)\rightarrow \so(\R^n)$ is the orthogonal projection, 
$\Har\pr \Har$ is the composition of the three endomorphisms
and 
\[
 \ml \Har^\#A,B \mr=-\tfrac{1}{2}\trace\bigl(\ad^{tr}_{\cdot}A\Har\ad^{tr}_{\cdot}B\Har\bigr).
\]
Then $C(t)$ defines a Ricci flow invariant condition, that  is, 
if $(M,g(t))$ ($t\in(0,T)$) is a compact  solution to the Ricci flow 
and $\Har(p,0)\in C(0)$ for all $p\in M$ 
then $\Har(p,t)\in C(t)$ for all  $t$.
\end{thm}

Of course maximum principles are well established in the literature including 
some for Harnack operators and generalization to open manifolds 
and the case of $t\to 0$ have been established. 
More important for us is that we only need $C$ to be 
invariant under the above representation 
of the relatively small group $\Iso(\R^n)$.

We consider the connection
\begin{eqnarray}\label{nabla}
 \nabla_XY=\nabla^{g(t)}_XY, \,\, 
\nabla_{\frac{\partial}{\partial t}}\tfrac{\partial}{\partial t}=-\tfrac{t}{2}\grad^{g(t)}(\scal)
\end{eqnarray}
\[
\nabla_{\frac{\partial}{\partial t}}Y=-\Ric^{g(t)}Y+
\tfrac{d}{dt}Y_{(p,t)},\,\,\nabla_Y{\frac{\partial}{\partial t}}=-t\Ric^{g(t)}Y-\tfrac{1}{2}Y
\]
for vectorfields $X,Y$ in $N= M\times [0,T)$ tangential to $M$. 
Notice that  $\nabla$ is neither  torsion free nor Riemannian with respect to a background metric.
However the distribution $TM\times [0,T)$ is parallel with respect to  
$\nabla$ and $\nabla$ respects the metric induced on this distribution by the 
background metric $\ml \cdot,\cdot\mr$.
The affine space $\tfrac{\partial}{\partial t}_{|(t,p)}+T_pM$ is also invariant
under the parallel transport with respect to $\nabla$.

The holonomy group of $\nabla$ is thus in a natural fashion isomorphic 
to a subgroup of $\Iso(T_pM)$. In fact for a closed curve $\gamma$ at $(p,t)$
in $N$ the parallel transport $\Par_\gamma$ is determined by the 
linear isometry $\Par_{\gamma|(T_pM,g(t))}$ and a translational part $\tau(\Par_{\gamma})\in T_pM$ 
characterized by $\Par_{\gamma}(\tfrac{\partial}{\partial t})=\tfrac{\partial}{\partial t}+\tau(\Par_{\gamma})$.
The map
\[
\Par_{\gamma}\mapsto (\Par_{\gamma|(T_pM,g(t))},\tau(\Par_{\gamma}))\in\Or(T_pM,g(t))\rtimes T_pM
\] 
is a homomorphism.

We identify $\iso(T_pM)=\Lambda^2T_pN=\so(T_pM)\oplus \{ \tfrac{\partial}{\partial t}\wedge v\mid v\in T_pM\}$ 
where 
we view $\so(T_pM)\cong \Lambda^2T_pM$ as the vector space of skew adjoint endomorphism
endowed with the scalar product $\ml X,Y \mr = -\tfrac{1}{2}\trace(XY)$, the second summand
$\tfrac{\partial}{\partial t}\wedge T_pM$ is orthogonal to $\so(T_pM)$
and the scalar product is given by $\ml \tfrac{\partial}{\partial t}\wedge v, \tfrac{\partial}{\partial t}\wedge w\mr =\ml v,w \mr$. 
The Lie bracket is given by
\[
[(X+\tfrac{\partial}{\partial t}\wedge v),Y+\tfrac{\partial}{\partial t}\wedge w)]=XY-YX+\tfrac{\partial}{\partial t}\wedge X w
-\tfrac{\partial}{\partial t}\wedge Y v.
\]

Notice that the holonomy group of $N$ with respect to $T_pN$ acts naturally on 
$\iso(T_pM)$. It is straightforward to check that this action corresponds to the coadjoint 
representation of $\Iso(T_pM)$ in $\iso(T_pM)$ given by $g\mapsto Ad^{tr}_{g^{-1}}$.

For $A,B\in \iso(T_pM)$ we define $\ad_A$ as usual $\ad_AB=[A,B]$ and let $\ad_A^{tr}$ 
denote the dual endomorphism and $\ad_{\cdot}^{tr}B$  the endomorphism
$A\mapsto  \ad_{A}^{tr}B$. It is straightforward to check that $\ad_{\cdot}^{tr}B$ 
is skew adjoint: $\ml \ad_{\cdot}^{tr}B(A),A\mr=\ml B,[A,A]\mr =0$. 

We extend the bilinear map $(A,B)\mapsto \ad_{A}^{tr}B$ 
to a complex bilinear map 
\[
\ad^{tr}\colon \iso(T_pM)\otimes_{\R}\C\times \iso(T_pM)\otimes_{\R}\C\rightarrow \iso(T_pM)\otimes_{\R}\C.
\]

Although we defined $\Har(p,t)$ as a self adjoint endomorphism 
of $\iso(T_pM,g(t))$, 
we view it as $(4,0)$-tensor in order to define
 $\Delta \Har$:
 Choose a basis
 $b_1,\ldots,b_k$ of the Lie algebra $\iso((T_pM,g(t)))$.  
For $v\in T_pM$ and small $s$ we define $b_i(\exp(sv))$ 
as the parallel extension of $b_i$ with  respect to the connection $\nabla$ on $N$ defined by \eqref{nabla} along the geodesic $\exp(sv)$ in $(M,g(t))$.
Then $\Delta\Har(p,t)$ is the selfadjoint endomorphism of 
$\iso(T_pM,g(t))$ characterized by 
\[
 \bml\Delta \Har(p,t)b_i,b_j\bmr=\sum_{k=1}^n\tfrac{d^2}{ds^2}_{|s=0}
\bml \Har(b_i(\exp(se_k))),b_j(\exp(se_k))\bmr.
\]
where $e_1,\ldots,e_n$ is an orthonormal basis of $(T_pM,g(t))$.

\begin{thm}\label{har evolution} $\Har$ satisfies the tensor identity 
\[
 \nabla_{\tfrac{\partial}{\partial t}}\Har=\Delta \Har+2(\Har\pr \Har +\Har^\#)
\]
 where $\nabla$ is the connection on $N$ defined by \eqref{nabla}  
and $(\Delta \Har)$ is defined as above.
\end{thm}
The proof of Theorem~\ref{har evolution} follows from Brendle [2009]. 
He derived a similar tensor identity for $\Ham$ using the 
following torsion free connection that is similar to the one introduced by Chow and Chu [1995].
\[
 D_XY= \nabla^{g(t)}_XY, \,\, 
D_{\frac{\partial}{\partial t}}\tfrac{\partial}{\partial t}=-\tfrac{1}{2}\grad^{g(t)}(\scal)-\tfrac{3}{2t}\tfrac{\partial}{\partial t}
\]
\[
D_{\frac{\partial}{\partial t}}Y=-\Ric^{g(t)}Y-\tfrac{Y}{2t}+
\tfrac{d}{dt}Y_{(p,t)},
\]
By Brendle the operator $\Ham$ satisfies 
the tensor identity
\begin{eqnarray}\label{evolution Ham}
 D_{\tfrac{\partial}{\partial t}}\Ham=\Delta \Ham+\tfrac{2}{t}\Ham+ 2(\Ham\pr \Ham +\Ham^\#)
\end{eqnarray}
We should mention that Brendle has a different but equivalent definition of the algebraic expression
$(\Ham\pr \Ham +\Ham^\#)$.
>From this equation Theorem~\ref{har evolution} follows by a straightforward calculation. 

The advantage of Theorem~\ref{har evolution} over \eqref{evolution Ham} 
is that the former  is nonsingular at $t=0$ and  since the connection is fairly 
natural with respect to $\ml\cdot,\cdot\mr$ it is easy to establish a dynamical version of 
the maximum principle. 
On the other hand $D$ has similar properties to $\nabla$ provided we endow $N$
with the background metric

\[g(v,w)=\tfrac{1}{t}g(t)(v,w),\,\,\, 
g(v,\tfrac{\partial}{\partial t})=0,\,\,\, g(\tfrac{\partial}{\partial t},\tfrac{\partial}{\partial t})=\tfrac{1}{t^{3}}
\]
The curvature tensor of $D$ is given by $\tfrac{1}{t}\Ham$.
Recently Cabezas-Rivas and Topping [2009] found a sequence $g^k$ of metrics on $N$ 
with the property that 
$g^{k}(v,w)=\tfrac{1}{t}g(t)(v,w)$ and $g^{k}(\tfrac{\partial}{\partial t},w)=0$ for $v,w\in T_pM$. 
The only constant
 $g^{k}(\tfrac{\partial}{\partial t},\tfrac{\partial}{\partial t})$
depending on $k$ diverges to infinity for $k\to \infty$. 
However the Levi Cevita connection of these metrics converge in the $C^\infty$ topology to
$D$. In particular the curvature tensor converges to $\tfrac{1}{t}\Ham$.
 Moreover $(M,g^k)$ is a Ricci soliton up to order $\tfrac{1}{k}$. 
 Cabezas-Rivas and Topping are then able to derive 
\eqref{evolution Ham} from the evolution of a curvature tensor under the Ricci flow.

We now turn to the proof of  Theorem~\ref{thm: max}.
Since $C(t)$ is invariant under the representation 
of $\Iso(\R^n)$ we can identify it naturally 
with a subset of $S^2_B(\iso(T_pM,g(s)))$ for all $(p,s)\in N$.  

We choose an auxiliary  smooth tensor field
 $T$ such that $T(p,t)$ is  a selfadjoint endomorphism
of $\iso(T_pM,g(t))$ representing 
an interior point of the closed convex set $C(t)$.

\begin{lem} Any tensor $\Har$ on $N$
satisfying 

\[
 \nabla_{\tfrac{\partial}{\partial t}}\Har=\Delta \Har+2(\Har\pr \Har +\Har^\#)
\]
can be approximated by a sequence $S_k$ of tensors
 on $M\times\bigl[\tfrac{1}{k},T-\tfrac{1}{k}]$ satisfying

\[
 \nabla_{\tfrac{\partial}{\partial t}}S_k=\Delta S_k+2(S_k\pr S_k +S^\#_k)+\eps_k(T-S_k)
\]
and $S(p,\tfrac{1}{k})$  represents an interior point of
$C(\tfrac{1}{k})$ and $\eps_k>0$ converges to $0$. 

\end{lem}

Clearly one can find an initial value $S(p,\tfrac{1}{k})\in Int(C\delta_k)$ 
such that $S(p,\tfrac{1}{k})-\Har(p,\tfrac{1}{k})$ ($p\in M$) converges to $0$ 
in the $C^\infty$ topology. 
Moreover, $S(p,\tfrac{1}{k})$ is a solution 
if and only if $S_k-\Har$ is a solution of an equation with the obvious modifications. 
Since one can prove similarly to Shi a priori estimates
for the corresponding linearized equation, it follows
that a solution of the initial value problem exists.

\begin{proof}[Proof of Theorem~\ref{thm: max}.]
Since $C(s)\subset S^2_B(\Lg)$ is invariant under $\Iso(\R^n)$ 
we can identify it naturally with a subset of 
$S^2_B(\iso(T_pM,g(t))$ for all $(p,t)\in N$.

It suffices to prove that $S_k(p,t)\in C(t)$ ($t\in \bigl[\tfrac{1}{k},T-\tfrac{1}{k}\bigr]$)
for a sequence  $S_k$ as in the lemma.
We assume, on the contrary, that for some minimal $t_0>\tfrac{1}{k}$ 
we can find some $p\in M$
such that $S_k(p,t_0)$ is contained in the boundary of $C(t_0)$. 

Because of the minimal choice of $t_0$ we know 
that $S_k(q,t)\in C(t)$ for all $t\le t_0$ and $q\in M$.
To get a contradiction we will show that $S_k(p,t_0-h)\not\in C(t_0-h)$ 
for small positive $h$.

For small $h\ge 0$ and $q\in M$ we 
define $H(s)\in S^2_B(\iso(T_q,g(t_0-h)))$
as the solution of the ODE
$H'=2(H\pr H+H^\#)$
with $H(0)=S_k(t_0-h)$. 
Since the family $C(t)$ is invariant under the ODE 
we know that $P_k(q,t_0-h):=H(h)\in C(t_0)$.
By construction
\[
 \nabla_{\tfrac{\partial}{\partial t}}P_k(p,t_0)=\Delta P_k(p,t_0)+\eps_k(T-P_k)(p,t_0)
\]

Using that $P_k(q,t_0)=S_k(q,t_0)\in C(t_0)$
and that $C(t_0)$ is invariant under the representation of 
$\Iso(\R^n)$ it is immediate that $\Delta P_k(p,t_0)=T_{P_k(p,t_0)}C(t_0)$.
Furthermore we know by construction that $\eps_k(T-P_k)(p,t_0)$ 
is contained in the interior of the tangent cone $T_{P_k(p,t_0)}C(t_0)$. 
We deduce that $P_k(p,t_0-h)\not\in C(t_0)$ for small positive $h$
-- a contradiction.
\end{proof}

\begin{rem}\label{rem: k}
\begin{enumerate}\item[a)] If one carries out everything in this section in the special case that 
$(M,g(t))$ is a K\"ahler manifold, then the holonomy group
of the connection $\nabla$ is isomorphic to a subgroup of $\gU(T_pM)\rtimes T_pM\subset \SO(T_pM)\rtimes T_pM$ 
and the image of the Harnack operator is contained in the Lie subalgebra $\Lg'$ of this group.
One can then formulate and prove an analogous statement for Harnack operators 
of  K\"ahler manifolds.
\item[b)] Let $\Har$ be a Harnack operator 
and  $R=\Har_{|so(n)}$.  A simple computation shows that the trace Harnack inequality 
is equivalent to \[\inf\{ \trace\bigl(\Ad_v\Har \Ad_v^{tr})\mid v\in \R^n\subset \Iso(\R^n)\} -\trace(R)\ge 0.\]
If $R$ has positive Ricci curvature, 
then there is a unique $v\in \R^n$ such that $\Ad_v\Har\Ad_v^{tr}$
has minimal trace.
\item[c)] The reason for the somewhat complicated approach toward the 
maximum principle is that it is in general not true that  
a convex set $C(t)$ is contained in the interior of another slightly larger convex 
set $C$ which is also invariant under the action of $\Iso(\R^n)$, cf. next subsection.
This is of course related to the fact that the connection is
 not compatible with a metric on the space of curvature tensors.
\end{enumerate}
\end{rem}

\subsection{Some negative results on Harnack inequalities}
It is elementary to check 
that the subspaces

\begin{eqnarray*}
V&:=&\bigl\{\Har\in S^2_B(\Lg)\mid \ml \Har(v),w\mr=0 \mbox{ for all $v,w\in\so(n)$}\bigr\}\\
W&:=&\bigl\{\Har\in S^2_B(\Lg)\mid \so(n)\subset \mathrm{kernel}(\Har)\bigr\}
\end{eqnarray*}
are invariant under the representation of $\Iso(\R^n)$.
The space $W$ can be characterized as the fixed point 
set of the normal subgroup $\R^n\subset \Iso(\R^n)$.

If 
the convex sets $C(t)$ in Theorem~\ref{thm: max} 
have the form $C'+V$ for some subset 
$C'\subset S^2_B(\Lg)$, then the condition $C(t)$ 
can only provide restrictions for the curvature tensor 
$R=\Har_{|\so(n)\times\so(n)}$.

\begin{lem} Let $C\subset S^2_B(\Lg)$ be a closed convex
set of maximal dimension which is invariant under $\Iso(\R^n)$.
Suppose that $C$ is not of the form $C'+V$.
After possibly replacing $C$ by $-C$ the following holds. 
For every $\Har\in C$ the restriction $\Har_{|\so(n)\times \so(n)}$ 
is a curvature operator with nonnegative Ricci curvature.
\end{lem}

\begin{proof}
Suppose, on the contrary, we can find $\Har_i\in C$ and $v_i\in \R^n$ 
such that the Ricci curvature $\Ric_i$ of $\Har_{i|\so(n)\times\so(n)}$
satisfies the following ($i=1,2$): $\Ric_1(v_1,v_1)>0$ and $\Ric(v_2,v_2)<0$.
Put $\Har_i(t)=\Ad_{tv_i}\Har_i\Ad^{tr}_{tv_i}$.
It is straightforward to check that the trace
$\trace(\Har_i(t))$ converges quadratically in $t$ to $(-1)^{i+1}\infty$ 
as $t\to \infty$ ($i=1,2$).
The element $X_i=\lim_{t\to\infty} \tfrac{\Har_i(t)}{t^2}$ 
is an element in the cone at infinity 
\[
\partial_{\infty}C:=\lim_{\lambda\to \infty}\tfrac{C}{\lambda}\]
of $C$.
We know $\trace(X_1)>0$, $\trace(X_2)<0$ and 
it is straightforward to check $X_i\in W$. 
Moreover the operator $X_i$ has $v_i\in \R^n$ 
in its kernel.

Clearly the cone $\partial_{\infty}C$ 
is invariant under the the representation of $\Iso(\R^n)$.
The barry center of the $\SO(n)$-orbit of $X_1$ (resp. $X_2$)
is a positive (resp. negative) multiple of the orthogonal projection 
of $\Lg$  to $\R^n$.

 Under $\SO(n)$ the vectorspace $W$ decomposes 
into a one dimensional trivial 
and an irreducible representation.
Since $X_i\in W$ itself is not a multiple of 
the orthogonal projection we can  
 also find 
a traceless operator $X\in \partial_\infty C\cap W$.
Clearly this implies $W\subset \partial_\infty C$.

The quotient space 
$V/W$ decomposes under $\SO(n)$ in two 
inequivalent irreducible nontrivial subrepresentations. 

Using that $C$ has maximal dimension 
we can find $\Har\in C$ such that for some 
$v\in \R^n$, $\bigl(\Ad_v\Har\Ad_v^{tr}-\Har\bigr)\in V$ projects 
to an element in $V/W$ which 
is not  contained in a nontrivial invariant subspace.

For each $t$ we choose $Y(t)\in W$ 
such that $L(t):=\Ad_tv\Har\Ad_tv^{tr}-Y(t)$
has minimal norm. 
It then follows that the norm of $L$ increases linearly 
and $L_\infty:=\lim_{t\to\infty}\tfrac{L(t)}{t}\in V\cap \partial_{\infty}C$ 
corresponds in $V/W$ to an element that does not lie in a nontrivial 
invariant subspace. 

This in turn shows $V\subset \partial_{\infty}C$ 
and hence $C$ is of the form $C'+V$.

\end{proof}

\begin{lem} Let $C\subset S^2_B(\Lg)$ be a convex subset
of maximal dimension which is invariant under the representation of $\Iso(\R^n)$. 
Suppose there is an element $\Har\in C$ such that 
$R:=\Har_{|\so(n)\times\so(n)}$ satisfies for some 
$v\in \R^n$: $\Ric(v,v)=0$
 and $R(\cdot,v,v,\cdot)\neq 0$. 
Then for any family of convex sets $C(t)$ 
which is invariant under the ODE, invariant under $\Iso(\R^n)$ 
with $C(0)=C$ we have 
$C(t)= C(t)+V$ for all $t>0$.
\end{lem}

The lemma shows for example that one can not 
prove a Harnack inequality in the class of
$3$-manifolds with positive Ricci curvature evolving under the Ricci flow
 by means of the maximum
principle of Theorem~\ref{thm: max}.

\begin{proof}
It is straightforward to check
that $X:=\lim_{t\to\infty}\tfrac{\Ad_{tv}\Har\Ad_{vt}^{tr}}{t^2}$
is a traceless operator in $W$. 
Clearly $X\in \partial_{\infty}C$.

Since the traceless operators $W'\subset W$ form an irreducible 
subspace  we deduce 
that $W'\subset \partial_{\infty}C$. 
Let $C(t)$ be as in the lemma.

Consider the element $-\id\in O(n)$
and let $\Fix(-\id)\subset S^2_B(\Lg)$ denote the fixed point 
set of $-\id$. Notice that $\Fix(-\id)$ is still
invariant under $\SO(n)$.
It is easy to see that $\Fix(-id)\cap C$ has maximal dimension
in $\Fix(-\id)$. 
Moreover the set 
$\tilde C(t)=C(t)\cap \Fix(-\id)$
is invariant under the ODE. 

Notice that $\tilde C(0)$ contains 
a subset of the form $\lambda I+ W'$ 
where $I$ is the orthogonal projection of $\Lg$ to $\so(n)$ 
for some $\lambda$. 
If we evolve this set under the ODE 
we see that $W'\subset \partial_{\infty}C(t)$ for all $t$. 

Thus $\tilde C(t)=\tilde C(t)'+W'$ 
where $\tilde C(t)'\subset W'^{\perp}\cap \Fix(-\id)$ 
is convex.

We may assume that the norm of $\Har_{|\so(n)\times \so(n)}$ 
is bounded by some a priori constant for all $\Har\in C(t)$. 
This in turn implies that a sequence in $\tilde C(t)'$ tends to 
$\infty$ if and only if its trace is unbounded.

Using that $\tilde C(t)$ is invariant under the ODE 
and that $C'(t)$ has full dimension
 it is easy to see that 
for all positive $t$  there are endomorphisms
with arbitrary small as well 
as endomorphisms with arbitrary large trace in $\tilde C(t)$ and hence in $\tilde C(t)'$.
Therefore  $W\subset \partial_\infty \tilde C(t)$. 
This implies as before $V\subset \partial_{\infty}C(t)$ 
for all positive $t$ as claimed.

\end{proof}

\section{Proof of Theorem~\ref{mainthm: har}.}\label{sec: mainthm har}
We prove a slightly more general result which holds for any metric Lie algebra.
Let $\Lg$ be a Lie algebra endowed with a scalar product $\ml\cdot,\cdot \mr$.
Put $\Lg_\C=\Lg\otimes_\R\C$ and let $\G$ and $\G_{\C}$ be associated groups.
For a self adjoint endomorphism $R\colon \Lg\rightarrow \Lg$
we define $R^\#\colon \Lg\rightarrow \Lg$ by 
\[
\ml R^\#v,w\mr= -\tfrac{1}{2}\trace \ad_{\cdot}^{tr}vR \ad_{\cdot}^{tr}wR \mbox{ for $v,w \in \Lg$.}
\]
Here $\ad_{\cdot}^{tr}v$ is the map $x\mapsto\ad_{x}^{tr}v$ which in turn is characterized by 
 $\ml\ad_{x}^{tr}v,y\mr = \ml v,[x,y]\mr$.
It is easy to see that  $\ad_{\cdot}^{tr}v$ is skew adjoint 
with respect to the scalar product for each fixed $v\in \Lg $.
If $\G$ is compact, one can choose $\ml\cdot,\cdot\mr$ to be  $\Ad_{\G}$ invariant  and 
then $\ad_{\cdot}^{tr}v=\ad_v$ and we get back the earlier definition of $R^\#$.

Each $R\in S^2(\Lg)$  induces a complex symmetric bilinear form 
on $\Lg_\C$ which we denote by $(x,y)\mapsto R(x,y)$. 
We extend the coadjoint representation $g\mapsto Ad_{g^{-1}}^{tr}$ with $\ml Ad_{g^{-1}}^{tr}x,y\mr= \ml x, Ad_{g^{-1}}y\mr$ 
to a representation of $\G_{\C}$ in $\Lg_{\C}$.

\begin{prop}
Consider the vectorspace $S^2(\Lg)$ of selfadjoint endomorphisms of $\Lg$ endowed with the ODE
  $R'= R^\#$ from above.
Suppose $S\subset \Lg_{\C}$ is invariant under the complexified coadjoint representation of $\G_{\C}$.
Then for each $h\in \R$ the  set 
\[
 C(S,h)=\{R\in S^2(\Lg)\mid R(x,\bar x)\ge h\mbox{ for all $s\in S$}\}
\]
 is invariant under the ODE $R'= R^\#$.
\end{prop}

Combining the proposition with the maximum principle from the previous section and with the fact that 
the first summand in the Harnack ODE is a nonnegative operator 
Theorem~\ref{mainthm: har} clearly follows.

\begin{proof}[Proof of the Proposition.] As before we have to show.\\[1ex]
{\bf Claim.} If $R\in C(S,h)$ and   $v\in S$ 
with $R(v,\bar v)=0$, then $R^\#(v,\bar v)\ge 0$. \\[1ex]
Up to some necessary changes in notation the proof is the same as the one in section~\ref{sec: ricci}.
For convenience we repeat it here in the more general setting.

We extend the maps $\Lg\times\Lg\rightarrow \Lg$, $(x,v)\mapsto \ad_x^{tr}v$ to a complex bilinear 
map $\Lg_{\C}\times\Lg_{\C}\rightarrow \Lg_\C$ which we also denote by $(x,v)\mapsto \ad_x^{tr}v$.

Using that $S$ is invariant under $\Ad_{\G_{\C}}^{tr}$ we deduce that for any $x\in \Lg\otimes_{\R}\C$ 
we have  for all $t\in\R$
\[h\le R(\Ad_{\exp(t x)}^{tr}v ,\Ad_{\exp(t \bar x)}^{tr}\bar v)\]
with equality at $t=0$. 
Recall that $\Ad_{\exp(t x)}^{tr}=\exp(t\ad_x^{tr})$. Thus differentiating twice 
with respect to  $t$ and evaluating at $0$ gives
\[
 0\le 2 R(\ad_x^{tr} v ,\ad_{\bar x}^{tr}\bar v) +R(\ad_x^{tr}\ad_x^{tr} v , \bar v)+R(v ,\ad_{\bar x}^{tr}\ad_{\bar x}^{tr}\bar v )
\]
If we now replace $x$ by $i x$, then it is easy to see that the first summand in the above inequality remains unchanged 
while the other two summands change their sign. 

Therefore 
\begin{eqnarray}\label{eq: nonneg*}
 0&\le& R(\ad_x^{tr} v ,\ad_{\bar x}^{tr}\bar v) \mbox{ for all $x\in \Lg\otimes_{\R}\C$. }
\end{eqnarray}
In other words,
$ -\ad_\cdot ^{tr} {\bar v} R\ad_\cdot^{tr} v$ and its conjugate $-\ad_\cdot ^{tr} {v} R\ad_\cdot^{tr} \bar v$
 are nonnegative hermitian
operators on the unitary vectorspace $\Lg\otimes_{\R}\C$.

In order to establish
$\trace (-\ad_\cdot^{tr} v R\ad_{\cdot}^{tr}\bar v R)\ge 0$,
 it now suffices to show 
that $R$ induces a nonnegative sesquilinear form on the image of the nonnegative operator
$-\ad_\cdot^{tr} v R\ad_{\cdot}^{tr}\bar v$. 
Clearly the image is contained in the image of 
$\ad_\cdot^{tr} v$ and by \eqref{eq: nonneg*} $R$ is indeed nonnegative on it which completes the proof of the proposition.
\end{proof}

\begin{rem}\begin{enumerate}\item[a)] Using Remark~\ref{rem: k} and the proposition in the case that 
$\Lg$ is given by the Lie algebra of $\gU(n)\rtimes \C^n$ one can derive Harnack inequalities 
for K\"ahler manifolds. 
The complexification of $\gU(n)\rtimes \C^n$  is given by 
$\GL(n,\C)\rtimes (\C^n\oplus \C^n)$, where $\GL(n,\C)$ acts in the standard way on the first summand 
and by $(A,v)\mapsto (\bar{A}^{tr})^{-1}v$ on the  second $\C^n$-summand.
Thus the set
\[S:=\{(A,v,0)\in \gl(n,\C)\times \C^n\times\C^n\mid \rank(A)=1, v\in A(\C^n)\},\]
is invariant under the coadjoint
representation. This gives a Harnack inequality for K\"ahler manifolds with positive
bisectional curvature whose trace form is similar to Cao's [1992] Harnack inequality.
\item[b)] Let $(\Lg,\ml \cdot,\cdot\mr) $ be as in the proposition and let $G\colon \Lg\rightarrow\Lg$
denote a selfadjoint positive endomorphism. Put $g(v,w)=\ml \cdot,G\cdot\mr$. 
The ODE $R'=R^{\#_g}$ corresponding to the metric Lie algebra 
$(\Lg,g)$ is 
obtained by pulling back the corresponding 
ODE for the metric Lie algebra $(\Lg,\ml \cdot,\cdot\mr)$ under the linear map 
$S^2(\Lg,g)\rightarrow S^2(\Lg,\ml\cdot,\cdot\mr), R\mapsto RG^{-1}$.
Thus $R^{\#_g}=(RG^{-1})^\#\cdot G$.
\end{enumerate}
\end{rem}

\section{Some pinching results}\label{sec: pinching}

Theorem~\ref{mainthm: ricci} gives a large family of invariant nonnegativity conditions. 
We will see in this section that it can also be used to show that some nonnegativity conditions
pinch toward stronger nonnegativity conditions.

\begin{lem}\label{lem: pinching} Consider, a $\Ad_{\G_\C}$-invariant subset $S\subset \Lg_{\C}$. 
Then
\[
\partial_\infty S:= \{ Y\in \Lg_{\C}\mid\mbox{ there exists $\lambda_i\in \R$ and $v_i\in S$ with $\lambda_i\to 0$ and 
$\lambda_iv_i \to Y$}\}
\]
is a scaling invariant $\Ad_{\G_{\C}}$-invariant set, which we call, 
by slight abuse of notation, the boundary of $S$ at infinity.

\begin{enumerate}
 \item[a)] For any $h\in \R$ the set $C(S,h)$ is contained in  $C(\partial_\infty S,0)$.
\item[b)] The union $\bigcup_{h<0}C(S,h)$ contains all interior points of 
$C(\partial_\infty S,0)$.
\end{enumerate}
\end{lem}

\begin{proof} {\em a).} Consider $R\notin C(\partial_\infty S,0)$. Then there is
 a $Y\in \partial_\infty S$ with $R(Y,\bar Y)=a <0$. 
Choose a sequence $v_i\in S$ and $\lambda_i\to 0$ with $\lambda_iv_i\to Y$. 
Then $R(v_i,\bar v_i)=-\tfrac{1}{\lambda_i^2}R(\lambda_iv_i,\lambda_iv_i)\to-\infty$.
Thus $R\notin C(S,h)$ for all $h$.

{\em b).} If $R$ is in the interior of $C(\partial_\infty S,0)$, 
then we can find a scaling invariant open neighborhood $U$ of $\partial_\infty S\setminus \{0\}$ 
such that $R(v,\bar v)> 0$ for all $v\in U$.
It is straightforward to check that the set $S'=\{v\in S\mid v\not\in U\}$ 
is bounded.
Thus if we put $h:=\inf_{v\in S'}R(v,\bar v)>-\infty$ we deduce $R\in C(S,\min\{0,h\})$.
\end{proof}

{\em Applications.} \begin{enumerate}
                     \item[a)] Nonnegative orthogonal bisectional  curvature pinches toward nonnegative bisectional curvature.
Let $(M,g(0))$ be  a compact K\"ahler manifold  with nonnegative orthogonal bisectional curvature:

Using a strong maximum principle 
it is not hard to see that under the Ricci flow either the orthogonal bisectional 
curvature turns positive immediately or the manifold is covered by a product or a symmetric space. 
A Hermitian symmetric space has nonnegative bisectional curvature and if $M$ is covered by a product one can argue for
each factor separately.
 Thus we may assume that $(M,g(t_0))$ has positive orthogonal bisectional curvature
and without loss of generality $t_0=0$.

We put $S=\{ X\in\gl(n,\C)\mid \trace(X)=\rank(X)=1 \}$.
It is easy to see that $\partial_\infty S$ is given by the space of nilpotent matrices
of $\rank \le 1$.
Thus $C(\partial_\infty S,0)$ corresponds to the space of curvature operators with nonnegative 
orthogonal bisectional curvature. 
>From the above Lemma we deduce that
$(M,g(t_0))$ satisfies the curvature condition $C(S,h)$ for some $h<<0$, that is 
$C(S,h)$ contains the compact set of curvature operators 
given by evaluating the curvature operator of $(M,g(0))$ at all base points. 

By Theorem~\ref{mainthm: ricci} we deduce that $(M,g(t))$ 
satisfies $C(S,h)$ for all $t$. 
This in turn implies that the bisectional curvature of $(M,g(t))$ stays bounded below by a fixed constant.
Since the scalar curvature blows up at a singularity this shows that 
$(M,g(t))$ pinches toward nonnegative bisectional curvature.

\item[b)]  
Let $L\in [0,\infty]$ and put 
\[
S(L):=\{X+zI\in \gl(n,\C)\mid  z\in \C,  |z|<L, \rank(X)=\trace(X)=1 \} 
\]
where $I$ denotes the identity matrix.
It is straightforward to check that 
$\partial_\infty S(L)$ is still given by the nilpotent rank 1 matrices, if $L<\infty$. 
Similarly to a) this in turn shows that nonnegative orthogonal bisectional curvature pinches 
toward the  curvature condition $C(S(\infty),0)$. 
In the case of $n=2$, $C(S(\infty),0)$ consists of the nonnegative K\"ahler curvature 
operators.

\item[c)] Suppose $n$ is even and put
$S=\{ X\in\so(n,\C)\mid X^2=-\id\}$. 
Then $\partial_\infty S=\{X\in\so(n,\C)\mid X^2=0\}$. 
As always $C(\partial_\infty S,0)$ pinches toward $C(S,0)$. 
If $n=4,6$, then $C(\partial_\infty S,0)$ coincides with nonnegative isotropic curvature.
In all even dimensions $n$
the manifold $\Sph^{n-1}\times\R$ satisfies $C(S,0)$ strictly.
\item[d)] Put $S=\{ X\in\so(n,\C)\mid\mbox{ the eigenvalues of $X$ have absolute value $\le 1$}\}$
Clearly $C(S,0)$ corresponds to the cone of nonnegative curvature operators.
Moreover $\partial_\infty S=\{ X\in\so(n,\C)\mid X^n=0\}$.
As before the curvature condition $C(\partial_\infty S,0)$ pinches toward $C(S,0)$. 
\end{enumerate}

\subsection{Manifolds satisfying PIC1.}\label{subsec: pic1}
Let $(M,g)$ be a compact manifold such that $\R\times M$ has positive isotropic curvature (PIC1).
Consider the subset $S\subset \so(n,\C)$ of rank 2
matrices with eigenvalues $\pm 1$. It is easy to see 
that $\partial_\infty S$ consists of all nilpotent matrices  $X$ in $\so(n,\C)$ of rank $\le 2$.
Moreover, such a matrix satisfies $X^3=0$, and if $X\neq 0$, then $\rank(X)=2$. 
Thus $(M,g)$ satisfies  $C(\partial_\infty S,0)$ strictly, see Remark~\ref{mainrem} d).
 By Lemma~\ref{lem: pinching} $(M,g)$ 
satisfies $C(S,h)$ for some $h<0$. 
By replacing $h$ by $h-1$ we may assume that $(M,g)$ satisfies $C(S,h)$ strictly.
We consider the linear map 
\[ 
l_s\colon S^2_B(\so(n))\rightarrow S^2_B(\so(n)),R\mapsto R+ 2s \Ric \wedge \id+(n-1)(n-2)s^2 R_I,
\]
where $R_I$ is the orthogonal projection of $R$ to multiples of the identity. 

Since  the  operators in $C(S,h)\subset C(\partial_\infty S,0)$ have nonnegative Ricci curvature,
we can apply [B\"ohm and Wilking 2008, Proposition 3.2], to see
that 
$l_s(C(S,h))$ defines a Ricci flow invariant curvature condition for positive $s\le \tfrac{\sqrt{2n(n-2)+4}-2}{n(n-2)}$. 

Clearly $(M,g(0))$ satisfies $l_s(C(S,h))$ for small $s>0$.
It is straightforward to check that the set $D= l_s(C(S,h))\setminus l_{s/2}(C(S,0))$ is bounded. 
Choose $L>0$ such that all operators in $D$ have trace $< L$.
We also assume that the scalar curvature of $(M,g)$ is bounded by $L$. 
Put
$K=\{ R\in l_{s/2}(C(S,0))\mid \trace(R)=L\}$. 
Notice that $K$ is a compact subset of the interior of $C(S,0)$. 
We now choose a convex, $\Or(n)$--invariant, ODE invariant set $F\subset l_{s/2}(C(S,0))$ with $K\subset F$ and 
$\lim_{\lambda\to \infty}\tfrac{1}{\lambda}F= \R_+I$. 
The existence of this set follows immediately from B\"ohm and Wilking [2008, proof of Theorem 3.1 combined with Theorem 4.1],
see [Theorem 6.1, Wilking 2007] -- here we used that  $C(S,0)$ is a Ricci flow invariant curvature 
condition in between nonnegative curvature operator and nonnegative sectional curvature.
We put 
\[
 \hat F := \Bigl(\{R\in F\mid \trace(R)\ge L\}\cap l_s(C(S,h))\Bigr) \cup \{R\in l_s(C(S,h))\mid \trace(R)\le L\}
\]
It is easy to see that $\hat F $ is convex and $O(n)$-invariant and ODE-invariant. 
Clearly $\lim_{\lambda\to \infty}\tfrac{1}{\lambda}\hat F=\lim_{\lambda\to \infty}\tfrac{1}{\lambda} F= \R_+I$.
Moreover by construction $(M,g)$ satisfies $\hat F $. By Theorem 5.1 in [B\"ohm and Wilking, 2008]
which is a slight extension of an earlier convergence result of Hamilton [1986],
$g$ evolves under the normalized Ricci flow to a constant curvature limit metric 
on $M$. 

This recovers the main theorem of [Brendle, 2008], which in turn generalized the main result of [Brendle and Schoen, 2009].

Part of the proof of Theorem~\ref{mainthm: bisec} in the next section 
is analogous to the above arguments. However, there are a few additional twists
which come from the fact that the ODE in the K\"ahler case behaves differently.

\section{K\"ahler manifolds with positive orthogonal bisectional curvature}\label{sec: bisec}
This section is devoted to the proof of Theorem~\ref{mainthm: bisec}. 

We first want to explain the equivalence of the two definitions we gave in the introduction.
We defined the cone of nonnegative orthogonal bisectional curvature 
as $C(S)$ where $S\subset \gl(n,\C)\cong \Lu(n)\otimes_\R\C$ is the space of nilpotent rank 1 
matrices. We claim $R\in C(S)$ if and only if  $K(v,w)+K(v,iw)\ge 0$ 
for all unit vectors $v,w$ satisfying $\C v\perp \C w$. 
Since a  complex rank one $n\times n$-matrix  is via an element in $\gU(n)$
conjugate to a matrix which is zero away from the upper $2\times 2$ block
it is clear that it suffices to explain the equivalence in the case of $n=2$. 
Using the natural embedding $\Lu(2)\subset \so(4)$ 
the nilpotent rank 1 matrices in $\gl(2,\C)$ correspond to 
totally isotropic rank 2 matrices in $\so(4,\C)$. 
A totally isotropic rank 2 matrix in $\so(4,\C)$ is contained in
an ideal $\su(2,\C)\subset \so(4,\C)$.    
It follows easily (for n=2) that $C(S)\cap S^2_K(\Lu(2))$ is given by the cone of those K\"ahler curvature operators in $S^2_K(\Lu(2))$
with nonnegative isotropic curvature. 

Let $R\in S^2_K(\Lu(2))$ with $n=2$ and 
$v,w\in \C^2\cong \R^4$ with $\C v\perp\C w$ then 
$R$ induces an endomorphism of $\so(4)$ such that 
$v\wedge w+ iw\wedge iv$ and
$v\wedge i w+ iv\wedge w$ are in the kernel of $R$. 
It is now easy to see that $K(v,iw)+K(v,w)=K(iv,w)+K(iv,iw)$ 
is a positive multiple of the isotropic curvature 
$R(v\wedge w- iw\wedge iv,v\wedge w- iw\wedge iv)+R(v\wedge i w- iv\wedge w , v\wedge i w- iv\wedge w)$. 
Thus nonnegativity of $K(v,iw)+K(v,w)$ for all possible choices $v$ and $w$
is equivalent to  $R\in C(S)$.

The Ricci flow on K\"ahler manifolds is particularly well behaved 
if the first Chern class is a multiple of the K\"ahler class. 
In our situation this will follow from 

\begin{lem}\label{lem: bochner} Let $M$ be a compact K\"ahler manifold with nonnegative orthogonal bisectional curvature. 
Then the Bochner operator on two forms is nonnegative. 
In particular any harmonic two form is parallel. 
If $(M,g)$ has positive orthogonal bisectional curvature, then $H^2(M,\R)\cong \R$.
\end{lem}
\begin{proof}
For a Riemannian manifold the Bochner operator on two forms is given 
by $\mathcal{R}:=\Ric\wedge \id -R$ where $R$ is the curvature operator 
$\Ric$ is the Ricci curvature and $\Ric\wedge \id (e_i\wedge e_j)=\tfrac{1}{2}\bigl(\Ric(e_i)\wedge e_j+ e_i\wedge \Ric(e_j)\bigr)$.
Compare for example [Ni and Wilking, 2009].
If $R$ is the curvature operator of a K\"ahler manifold, 
then it is easy to see that $\Ric \wedge\id$ leaves the Lie algebra $\Lu(n)\subset \so(2n)$ 
invariant. Since the orthogonal complement of $\Lu(n)^\perp$
is contained in kernel of $R$, we can show $\mathcal{R}_{| \Lu(n)^\perp}\ge 0$ by establishing\\[1ex] 
{\bf Claim 1.} For a K\"ahler manifold with nonnegative orthogonal bisectional curvature
the following holds: If  $v_1,v_2\in T_pM$ are unit vectors with $\C v_1\perp \C v_2$, 
then $\Ric(v_1,v_1)+ \Ric(v_2,v_2)\ge 0$. \\[1ex]
We extend $v_1,v_2$ to a complex orthonormal basis $v_1,\ldots,v_n$ of $T_pM$. 
Then
\begin{eqnarray*}
\Ric(v_1,v_1)+ \Ric(v_2,v_2)&=&K(v_1,iv_1)+2K(v_1,v_2)+2K(v_1,iv_2)+K(v_2,iv_2)+\\&&+\sum_{j=3}^n K(v_1,v_j)+K(v_2,v_j)+K(v_1,iv_j)+K(v_2,i v_j)
\end{eqnarray*}

Since $R$ has nonnegative orthogonal bisectional curvature 
we know $K(v_j,v_k)+K(v_j,iv_k)\ge 0$ for $k\neq l$. 
In other words it suffices to show that the first four summands add up to a nonnegative number. 
This in turn is equivalent to establishing the claim for complex surfaces. 
But if $n=2$, then $\Ric(v_1,v_1)+ \Ric(v_2,v_2)=\tfrac{1}{2}\scal\ge 0$.

We can finish the proof of the first part of Lemma by establishing\\[1ex]
{\bf Claim 2.} $\mathcal{R}_{| \Lu(n)}\ge 0$.\\[1ex]
Suppose $\omega\in \Lu(n)$ is an eigenvector of $\mathcal{R}$. 
We can find an orthonormal basis $v_1,\cdots, v_n$ with $\C v_j\perp \C v_k$ for $j\neq k$
and real numbers $\lambda_i$ such that $\omega$ is given by 
$\omega=\sum_{j=1}^n \lambda_j v_j\wedge iv_j$.
One now checks by straightforward computation
\[
 2\ml\mathcal R \omega,\omega\mr=\sum_{j\neq k}^n (\lambda_j^2+\lambda_k^2)\bigl( K(v_j,v_k)+ K(v_j,iv_k)\bigr)
-2\sum_{j\neq k}\lambda_j\cdot \lambda_k R(v_j, iv_j,v_k,iv_k)
\]
Since $v_j\wedge v_k+ iv_k\wedge iv_j$ and
$v_j\wedge i v_k+ iv_j\wedge v_k$ are in the kernel of $R$ it follows from the first Bianchi identity
that $R(v_j, iv_j,v_k,iv_k)=K(v_j,iv_k)+ K(v_j,v_k)$. Hence
\[
 2\ml\mathcal R \omega,\omega\mr=\sum_{j\neq k}^n (\lambda_j-\lambda_k)^2\bigl(K(v_j,v_k)+ K(v_j,iv_k)\bigr)
\]
which is nonnegative as each summand is nonnegative.

This shows that $\mathcal R$ is nonnegative. Therefore any harmonic two form is parallel. 
If the orthogonal bisectional curvature is positive, it is easy to deduce that 
the kernel of $\mathcal R$ is given by multiples of the K\"ahler form and thus any harmonic two form is a multiple of
 the K\"ahler form.
\end{proof}

As before we consider the Lie algebra $\Lu(n)$ 
of skew hermitian $n\times n$ matrices endowed with the scalar product
$\ml u,v\mr=-\trace{u\cdot v}$. 
The vectorspace of K\"ahler curvature operators $S^2_K(\Lu(n))$ 
 can be naturally seen as a subspace 
of the space  $S^2(\Lu(n))$ of selfadjoint endomorphism of $\Lu(n)$. 
 
Given two hermitian endomorphisms $A,B\colon \C^n\rightarrow \C^n$
we let $A\star B\colon \Lu(n)\rightarrow \Lu(n)$ 
denote the self adjoint endomorphism of $\Lu(n)$ defined by 
\[
2\ml  A\star B u,v\mr =-\trace AuBv-\trace(AvBu)-\trace(Au)\trace(Bv)-\trace(Av)\trace(Bu).
\]
A straightforward computation shows that 
$A\star B$ is a K\"ahler curvature operator. 
We put $E=\id\star \id$. 
Then $E$ corresponds to the curvature operator of $\CP^n$
scaled such that the sectional curvature lies in the interval [1/2,2]. Thus
$E$ has the eigenvalue $1$ with multiplicity $n^2-1$ and the eigenvalue $n+1$ 
with multiplicity $1$. 
The operators of the form $A\star \id$ are precisely given by the orthogonal complement 
of the Ricci flat 
operators $\ml W\mr$ in $S^2_K(\Lu(n))$.

For $R\in S^2_K(\Lu(n))$ we let $\Ric(R)$ denote its Ricci curvature 
which we can view as a hermitian $n\times n$ matrix. 
We define a linear map 
\begin{eqnarray}\label{eq: lt}
l_{s}\colon S^2_K(\Lu(n))\rightarrow S^2_K(\Lu(n)),\,\, R\mapsto R+2s\Ric(R)\star\id +s^2\scal(R) E.
\end{eqnarray}

Similarly, to [B\"ohm and Wilking, 2008] we are interested in how the Ricci flow ODE 
changes if we pull it back under $l_s$.
It is not hard but tedious to derive a formula similar to the one in [B\"ohm and Wilking, 2008].
However, for our purposes here the following simple formula will be sufficient. 
\begin{lem}\label{lem: kaehler} For $R\in S^2_K(\Lu(n))$ put 
$D(s)(R)=l_s^{-1}(l_s(R)^2+l_s(R)^\#)-R^2-R^\#$. 
Then 
\[
 \tfrac{d}{ds}_{|s=0} D(s)(R)=D'(0)(R)=2\Ric(R)\star \Ric(R)
\]
\end{lem}
\begin{proof}
We let $\ml W\mr\subset S^2_K(\Lu(n))$ denote the kernel of
$R\mapsto \Ric(R)$. For $R\in S^2_K(\Lu(n))$ the orthogonal projection $R_W$ 
of $R$ to $\ml W\mr $ is called the (K\"ahler)-Weyl part of $R$. 
As in the real case $\ml W\mr$ is an irreducible module 
and analogously to [B\"ohm and Wilking, 2008] one can show that $D(t)(R)$ is independent of $R_W$. 
We let \[B(s)(R_1,R_2)=\tfrac{1}{4}\bigl(D(s)(R_1+R_2)-D(s)(R_1-R_2)\bigr)\] denote the corresponding bilinear form.
Since any Ricci tensor is the sum of commuting rank one tensors 
it suffices to prove the corresponding statement  
for $B$ in the special case that $\Ric(R_1)$ and $\Ric(R_2)$ are commuting rank 1 matrices. 
Clearly we may assume that $\Ric(R_1)$ have $1$ as an eigenvalue. 
Using the polarization
\[B(s)(R_1,R_2)=\tfrac{1}{2}\bigl(D(s)(R_1+R_2)-D(t)(R_1)-D(s)(R_2)).\] 
we deduce that it suffices 
to prove the original statement for the following two special cases:
The statement holds for one curvature operator for which $\Ric(R)$ has rank 1 
 and the statement holds for one curvature operator for which
$\Ric(R)$ has rank 2 with 2 equal nonzero eigenvalues. 
In particular, it suffices to check that the statement holds 
in the case that $R$ is the curvature operator $R_k$ of 
$\CP^k\times \C^{n-k}$, $k=1,\ldots,n$.

It is straightforward to check that 
$D'(0)(R)$ and $2\Ric\star \Ric$ have the same trace.
Clearly this proves the formula for $R_k$ in the case $k=n$. 

Notice that 
$R_k^2+R_k^\#=(k+1)R_k$.
It is easy to see that $D'(0)(R_k)_{|\Lu(k)}$ is independent of $n$. 
Using that we know the formula in the case $k=n$ we deduce 
$D'(0)(R_k)_{|\Lu(k)}=2\Ric\star \Ric_{|\Lu(k)}$. 
Moreover it is easy to see that both operators 
contain the subalgebra $\Lu(n-k)$ in their kernel. 
It remains to check that  $D'(0)(R_k)$ restricted to $(\Lu(n-k)\oplus \Lu(k))^\perp$ 
vanishes. For symmetry reasons this restriction is given by 
a multiple of the identity. 
Combining this with the facts that  $D'(0)(R_k)$ and $2\Ric\star \Ric$ 
coincide on $\Lu(n-k)\oplus \Lu(k)$ and have the same trace the lemma follows.
\end{proof}

\begin{defin} We consider two subsets 
$C_1,C_2\subset S^2_K(\Lu(n))$ which are convex, closed and $\gU(n)$-invariant. 
Recall that we say that a K\"ahler manifold $(M,g)$ satisfies $C_i$ if the curvature operator
at each point is contained in $C_i$.
We say that $C_1$ defines a Ricci flow invariant curvature condition under 
the constraint $C_2$ if the following holds:
Any compact solution $(M,g(t))$ to the unnormalized K\"ahler Ricci flow $(t\in [0,T])$
satisfying $C_2$ at all times and satisfying $C_1$ at $t=0$, satisfies 
$C_1$ at all times.  
\end{defin}

One can carry over [see Chow and Lu, 2004] the proof of Hamilton's maximum principle 
to show
\begin{thm} Suppose for all $R\in C_1\cap C_2$ we have $R^2+R^\#\in T_RC_1$, 
then $C_1$ defines a Ricci flow invariant curvature condition under the constraint $C_2$.
\end{thm}

As consequence of this and the previous Lemma we obtain
\begin{cor}\label{cor: pinching} Suppose $C\subset S^2_K(\Lu(n))$ is a convex $\gU(n)$-invariant set 
which is invariant under the ODE $R'=R^2+R^\#$ and contains the space of nonnegative K\"ahler curvature 
operators. Let $p\in (0,1)$ and put
\[
C_2(p)=\{R\in S^2_K(\Lu(n))\mid \Ric(R)\ge p\tfrac{\scal}{2n}\}
\]
Then there is an $s_0=s_0(p,C)>0$ such that  the set $l_s(C)$ defines a Ricci flow invariant curvature condition under
the constraint $C_2(p)$ for all 
$s\in [0,s_0]$, where $l_s$ is the linear map defined by \eqref{eq: lt}.
\end{cor}

\begin{proof}
Put $X(s)(R)=l_{s}^{-1}(l_{s}(R)^2+l_{s}(R)^\#)$.
By the above Lemma 
\[
X(s)=R^2+R^\#+ 2s \Ric(R)\star \Ric(R) +O(s^2)
\]
where $O(s^2)$ stands for an operator satisfying 
$\|O(s^2)\|\le C s^2\|\Ric(R)\|^2$ for some constant $C>0$, $s\in [0,1]$.

We choose $s_0$ so small that $l_{s}^{-1}(C_2(p))\subset C_2(p/2)$ for all $s\in [0, s_0]$. 
Thus it suffices to check that 
$X(s)(R)\in T_RC$ for all $R\in C\cap C_2(p/2)$ and for  all small $s$.
Using  our estimate on $O(s^2)$ we can find $s_0$ such that 
the operator $2s\Ric(R)\star \Ric(R)+O(s^2) $ is positive for all $R\in C\cap C_2(p/2)$ and all $s\in (0,s_0]$.
Since the positive operators are contained in $C$ they are also contained in the 
tangent cone $T_RC$. 
By assumption $R^2+R^\#\in T_RC$ and thus $X(s)(R)\in T_RC$ for all $R\in C\cap C_2(p/2)$.

\end{proof}

\begin{proof}[Proof of Theorem~\ref{mainthm: bisec}]
We consider a solution
to the unnormalized Ricci flow $(M,g(t))$, $t\in [0,T)$. Since the scalar curvature is positive ($n\ge 2$),
 a finite time singularity $T$ occurs.
By Lemma~\ref{lem: bochner} we have $H^2(M,\R)\cong \R$. Thus the first Chern class is 
a multiple of the K\"ahler class. Since a finite time singularity occurs this in turn implies 
it is positive.\\[1ex] 
{\bf Claim.} For some $\eps, p>0$ we have that $(M,g(t))$ satisfies $C_2(p)$ (see Corollary~\ref{cor: pinching}
for a definition) for all $t\in [T-\eps,T)$.\\[1ex]
By Lemma~\ref{lem: pinching} we can find some small $h$ such that $(M,g(t))$ 
satisfies $C(S,h)$, where $S\subset \gl(n,\C)$ is the set of rank 1 matrices which have $1$ as an eigenvalue.

We argue by contradiction and assume that we can find $p_i\to 0$ and $t_i\to T$ 
such that $(M,g(t_i))$ does not satisfy $C_2(p_i)$. 
We rescale the manifold to have maximal curvature one. 
By an argument of Perelman which was written up by Sesum and Tian [2006] 
$(M,\lambda_ig(t_i))$ 
subconverges to a compact  limit manifold $(M,g_\infty)$.
 
$(M,g_\infty)$ is a K\"ahler manifold satisfying the curvature condition 
$\lim_{i\to \infty} \tfrac{1}{\lambda_i} C(S,h)=C(S,0)$. 
Recall that $C(S,0)$ is the cone of curvature operators with nonnegative bisectional curvature. 
Thus $(M,g_{\infty})$ has nonnegative  bisectional curvature
and in particular nonnegative Ricci curvature. By compactness 
we can assume that $g_\infty$ has a backward solution to the Ricci flow
with nonnegative bisectional curvature.
Since $(M,g_{\infty})$ is diffeomorphic to $M$ we know 
from Lemma~\ref{lem: bochner} that its second homology is isomorphic to $\R$. 
Thus $(M,g_\infty)$ does not have any flat factors. 
Combining with the strong maximum principle we deduce that $(M,g_\infty)$ has positive Ricci curvature. 
But this contradicts our choice of $(M,g(t_i))$.

After replacing $g(0)$ by $g(T-\eps)$ we may assume that
$(M,g(t))$ satisfies $C_2(p)$ for all $t\in [0,T)$. 
Moreover we can assume that the curvature operator of $(M,g(0))$ at each point 
is contained in the interior of $C(S,h)$ -- otherwise one can just replace $h$ by $h-1$.

This in turn shows that  $(M,g(0))$ satisfies $l_s(C(S,h))$ for  sufficiently small $s>0$. 
By Corollary~\ref{cor: pinching}  $(M,g(t))$ satisfies $l_s(C(S,h))$ for all $t\in [0,T]$ and some $s>0$. 
It is elementary to check that 
there is an $\eps>0$ and $C>0$ such that for all $R\in l_s(C(S,h))$
we have 
\[
 R  -(\eps \scal(R)-C)E\in C(S,0)
\]
Since for the unnormalized Ricci flow the scalar curvature of $(M,g(t))$ converges uniformly to $\infty$ for $t\to T$,
we deduce that  $(M,g(t))$ has positive bisectional curvature for some $t$.

\end{proof}

\section*{Appendix: Strong maximum principle for the Ricci flow.}

In this appendix we will sketch the argument for the following extension of Brendle and Schoen's
maximum principle.
\begin{thm} Let $S\subset \so(n,\C)$ be an $\Ad_{\SO(n,\C)}$--invariant subset 
and consider a solution to the Ricci flow $(M,g(t))$, $t\in [0,T)$, satisfying
$C(S)$ for all $t$.
By choosing a linear isometry between $(T_pM,g(t))$ and $\R^n$
we obtain a subset $S(p,t)\subset \so((T_pM,g(t)))\otimes_\R\C$
corresponding to $S$ for each $(p,t)$.
Put 
\[
 N(p,t)=\{X\in S(p,t)\mid R_{g(t)}(X,\bar X)=0\}
\]
Then $N(p,t)$ is invariant under parallel transport for $t>0$.
\end{thm}
As usual with strong maximum principles
we do not require that $(M,g(t))$ is compact or complete. 
The methods used to derive the above theorem from Theorem~\ref{mainthm: ricci} (and its proof) 
are due to Brendle and Schoen [2008]. 
In fact,
 Proposition~8 
in that paper is the special case of the above theorem
where $S$ is given by the totally isotropic rank 2 matrices.
The analogue of the above theorem for K\"ahler manifolds holds as well with the same proof.

A delicate part of Brendle and Schoen's proof of the strong maximum principle for isotropic curvature
is that it does not just use the invariance of positive isotropic curvature 
but also the proof of the invariance by means of  first and second variation formulas.
This is here true as well.

\begin{proof}
In the following we can assume that $S$ is invariant under scaling with positive numbers. 
Moreover we may assume that for $X$, $Y\in S$ there is some 
$g\in \SO(n,\C)$ and $\lambda>0$  with $\Ad_gX=\lambda Y$. 
In fact otherwise we decompose $S$ into subsets with this property and prove the theorem 
for each subset separately. 
Notice that these assumptions imply in particular that $S$ is a submanifold
with a transitive smooth action of $\SO(n,\C)\times \R_+$. 
Therefore $S(p,t)$ defines a bundle over $M\times (0,T)$
whose totalspace we denote by $T$.
We consider on $T$ the function 
\[ u\colon T\rightarrow \R,
\,\,  u(v)= R_{g(t)}(v,\bar v).\]

We lift the vectorfield $\tfrac{\partial}{\partial t}$ horizontally to a vectorfield on 
$T$ using the connection induced on $T$ by the connection  $\nabla$ on $M\times (0,T)$ from section~\ref{sec: har}.
We denote this horizontal lift again by $\tfrac{\partial}{\partial t}$.
Then

\[
\tfrac{\partial u}{\partial t}(v,\bar v)= \Delta_h u+ 2(R^2(v,\bar v)+ R^\#(v,\bar v)),
\]
where $\Delta_h u$ is the horizontal Laplacian which is defined 
as follows. Choose in a neighborhood of $(p,t)$ 
vectorfields $X_1,\ldots, X_n$ tangential to $M$ with $g(t)(X_i,X_j)=\delta_{ij}$ 
and put $Y_j=\nabla_{X_i}^{g(t)}X_i$. 
Let $\hat X_i$ and $\hat Y_i$ denote the horizontal lifts 
of $X_i$ and $Y_i$ to $T$. Then $\Delta_h u= \sum_{i=1}^n\hat X_i \hat X_iu-\hat{Y}_iu$.

From the proof of the invariance in section~\ref{sec: ricci} we can derive 
\[
R^\#_{g(t)}(v,\bar v)\ge C\|R\| \inf\bigl\{\tfrac{d^2}{dt^2} u(Ad_{\exp(tx)}v) \bigm| x\in \so((T_pM,g(t)),\C),\|x\|\le 1\bigr\}
 \]
for all $v\in T_pM$ where $C=C(n)$ is a constant.
We now introduce coordinates  on some relative compact subset $U\subset T$, which have an extension to 
 a neighborhood of $\bar U$. 
Then  the corresponding function $\tilde u$ in local coordinates, which is defined on some open subset $V\subset \R^k$, satisfies 
\[
\sum_{i= 1}^n \tilde X_i\tilde X_i \tilde u\le -K\inf\bigl\{\Hess(\tilde u)(a,a)\mid a\in\R^n, \|a\|\le 1\bigr\} +K\|\grad(\tilde u)\|
\]for some large constant $K$,
where $\tilde X_i$ denote the corresponding vectorfields in coordinates. 
We can now apply Proposition~4 from [Brendle and Schoen, 2008], to see 
that the level set $\tilde{u}^{-1}(0)$ is invariant under the (local) flows of the vectorfields $X_i$. 
Translating this back we obtain that the level set $u^{-1}(0)$ in $T$ is  
invariant under spacial parallel translation.
\end{proof}

\hspace*{1em}\\
\begin{footnotesize}
\hspace*{0.3em}{\sc University of M\"unster,
Einsteinstrasse 62, 48149 M\"unster, Germany}\\
\hspace*{0.3em}{\em E-mail addresses: }{\sf wilking@math.uni-muenster.de}
\end{footnotesize}

\end{document}